\documentclass[12pt]{article}

\usepackage[english]{babel}
\usepackage[dvipsnames]{xcolor}
\usepackage[colorlinks]{hyperref}

\usepackage{geometry}
\usepackage{enumitem}
\usepackage{seqsplit}
\usepackage{xstring}
\setenumerate{label=(\roman*)}
\makeatletter
\let\@subjclass\@empty
\let\@subjclassyear\@empty
\let\@keywords\@empty
\newcommand{\subjclass}[2][2020]{%
  \gdef\@subjclassyear{#1}%
  \gdef\@subjclass{#2}%
}
\newcommand{\keywords}[1]{\gdef\@keywords{#1}}
\let\article@maketitle\maketitle
\renewcommand{\maketitle}{%
  \article@maketitle
  \begingroup
  \renewcommand{\thefootnote}{}%
  \ifx\@subjclass\@empty\else
    \begin{NoHyper}
      \footnotetext{\textit{\@subjclassyear\ Mathematics Subject Classification.}
        \@subjclass.}%
    \end{NoHyper}
  \fi
  \ifx\@keywords\@empty\else
    \begin{NoHyper}
      \footnotetext{\textit{Key words and phrases.} \@keywords.}%
    \end{NoHyper}
  \fi
  \endgroup
}
\makeatother

\newcommand\myshade{100}
\hypersetup{
  linkcolor  = RoyalBlue!\myshade!black,
  citecolor  = ForestGreen!\myshade!black,
  urlcolor   = RedOrange!\myshade!black,
  colorlinks = true,
}

\DeclareRobustCommand{\faGithub}{{\usefont{U}{fontawesomebrands0}{regular}{n}\char167}}
\DeclareRobustCommand{\faFolderOpen}{{\usefont{U}{fontawesomefree1}{solid}{n}\char105}}
\DeclareRobustCommand{\faFileCode}{{\usefont{U}{fontawesomefree1}{solid}{n}\char68}}
\colorlet{githubrefcolor}{RedOrange!\myshade!black}
\makeatletter
\newcommand{\github@breakable@text}[1]{%
  \begingroup
  \edef\github@breakable@result{#1}%
  \texttt{\expandafter\seqsplit\expandafter{\github@breakable@result}}%
  \endgroup
}
\newcommand{\github@breakable@literal}[1]{%
  \begingroup
  \edef\github@text{\detokenize{#1}}%
  \github@breakable@text{\github@text}%
  \endgroup
}
\newcommand{\github@ref}[3]{%
  \href{#3}{\textcolor{githubrefcolor}{#1\,#2}}%
}
\newcommand{\github@ref@literal}[3]{%
  \href{#3}{\textcolor{githubrefcolor}{#1\,\github@breakable@literal{#2}}}%
}
\newcommand{\github@faicon}[1]{{\scriptsize #1}}
\DeclareRobustCommand{\leanicon}{%
  \begingroup
  \scriptsize
  \leavevmode
  \ooalign{%
    \hfil$\bm\forall$\hfil\cr
    \hfil\kern0.018em$\bm\forall$\hfil\cr
    \hfil\kern0.036em$\bm\forall$\hfil\cr
    \hfil\kern-0.018em$\bm\forall$\hfil\cr
    \hfil\kern-0.036em$\bm\forall$\hfil\cr
  }%
  \endgroup
}
\newcommand{\github@strip@suffix}[1]{%
  \IfSubStr{#1}{?}{\StrBefore{#1}{?}[#1]}{}%
  \IfSubStr{#1}{\#}{\StrBefore{#1}{\#}[#1]}{}%
}
\newcommand{\github@repo@text}[1]{%
  \begingroup
  \StrBehind{#1}{github.com/}[\github@path]%
  \github@strip@suffix{\github@path}%
  \IfSubStr{\github@path}{/}{%
    \StrBehind{\github@path}{/}[\github@afterowner]%
    \IfSubStr{\github@afterowner}{/}%
      {\StrBefore{\github@afterowner}{/}[\github@text]}%
      {\let\github@text\github@afterowner}%
  }{%
    \let\github@text\github@path
  }%
  \github@breakable@text{\github@text}%
  \endgroup
}
\newcommand{\github@path@text}[2]{%
  \begingroup
  \IfSubStr{#1}{/#2/}{%
    \StrBehind{#1}{/#2/}[\github@path]%
    \github@strip@suffix{\github@path}%
    \IfSubStr{\github@path}{/}%
      {\StrBehind{\github@path}{/}[\github@text]\github@breakable@text{\github@text}}%
      {\github@repo@text{#1}}%
  }{%
    \def\github@text{#1}%
    \github@breakable@text{\github@text}%
  }%
  \endgroup
}
\newcommand{\github@line@text}[1]{%
  \begingroup
  \StrBehind{#1}{/blob/}[\github@path]%
  \github@strip@suffix{\github@path}%
  \IfSubStr{\github@path}{/}%
    {\StrBehind{\github@path}{/}[\github@file]}%
    {\StrBehind{#1}{github.com/}[\github@file]\github@strip@suffix{\github@file}}%
  \StrBehind{#1}{\#}[\github@line]%
  \IfSubStr{\github@line}{?}{\StrBefore{\github@line}{?}[\github@line]}{}%
  \github@breakable@text{\github@file}\space\github@breakable@text{\github@line}%
  \endgroup
}
\newcommand{\github@repo@default}[1]{%
  \github@ref{\github@faicon{\faGithub}}{\github@repo@text{#1}}{#1}%
}
\newcommand{\github@folder@default}[1]{%
  \github@ref{\github@faicon{\faFolderOpen}}{\github@path@text{#1}{tree}}{#1}%
}
\newcommand{\github@file@default}[1]{%
  \github@ref{\github@faicon{\faFileCode}}{\github@path@text{#1}{blob}}{#1}%
}
\newcommand{\github@line@default}[1]{%
  \github@ref{\leanicon}{\github@line@text{#1}}{#1}%
}
\newcommand{\github@tree@default}[1]{%
  \begingroup
  \StrBehind{#1}{/tree/}[\github@treepath]%
  \github@strip@suffix{\github@treepath}%
  \IfSubStr{\github@treepath}{/}%
    {\endgroup\github@folder@default{#1}}%
    {\endgroup\github@repo@default{#1}}%
}
\newcommand{\github@auto@default}[1]{%
  \IfSubStr{#1}{/blob/}{%
    \IfSubStr{#1}{\#L}%
      {\github@line@default{#1}}%
      {\github@file@default{#1}}%
  }{%
    \IfSubStr{#1}{/tree/}%
      {\github@tree@default{#1}}%
      {\github@repo@default{#1}}%
  }%
}
\newcommand{\github@auto@literal}[2]{%
  \IfSubStr{#2}{/blob/}{%
    \IfSubStr{#2}{\#L}%
      {\github@ref@literal{\leanicon}{#1}{#2}}%
      {\github@ref@literal{\github@faicon{\faFileCode}}{#1}{#2}}%
  }{%
    \IfSubStr{#2}{/tree/}%
      {\github@ref@literal{\github@faicon{\faFolderOpen}}{#1}{#2}}%
      {\github@ref@literal{\github@faicon{\faGithub}}{#1}{#2}}%
  }%
}
\NewDocumentCommand{\gh}{om}{%
  \IfNoValueTF{#1}%
    {\github@auto@default{#2}}%
    {\github@auto@literal{#1}{#2}}%
}
\makeatother

\usepackage{amsmath}
\usepackage{amssymb}
\usepackage{amsthm}
\usepackage{bm}
\usepackage{bbold}

\usepackage[no-math]{fontspec}
\usepackage[capitalise]{cleveref}
\crefformat{equation}{(#2#1#3)}
\crefname{subsection}{Section}{Sections}

\theoremstyle{plain}
\newtheorem{thm}{Theorem}[section]
\newtheorem{lem}[thm]{Lemma}
\newtheorem{prop}[thm]{Proposition}
\newtheorem{cor}[thm]{Corollary}

\theoremstyle{definition}
\newtheorem{defn}[thm]{Definition}

\theoremstyle{remark}
\newtheorem*{rem}{Remark}

\ExplSyntaxOn
\NewDocumentCommand{\DeclareLeanTheoremEnvironment}{m}
  {
    \cs_new_eq:cc { lean_old_#1: } { #1 }
    \cs_new_eq:cc { lean_old_end_#1: } { end#1 }
    \RenewDocumentEnvironment{#1}{ o d<> }
      {
        \IfNoValueTF{##1}
          { \use:c { lean_old_#1: } }
          { \use:c { lean_old_#1: } [##1] }
        \IfNoValueF{##2}
          { \nobreak\textup{##2}\space }
      }
      { \use:c { lean_old_end_#1: } }
  }
\ExplSyntaxOff

\DeclareLeanTheoremEnvironment{thm}
\DeclareLeanTheoremEnvironment{lem}
\DeclareLeanTheoremEnvironment{prop}
\DeclareLeanTheoremEnvironment{cor}
\DeclareLeanTheoremEnvironment{defn}

\makeatletter
\def\@fnsymbol#1{\ensuremath{\ifcase#1\or \dagger\or \ddagger\or
           \dagger\dagger
           \or \ddagger\ddagger \else\@ctrerr\fi}}
\makeatother
\renewcommand{\thefootnote}{\fnsymbol{footnote}}

\newcommand\N{\ensuremath{\mathbb{N}}}
\newcommand\R{\ensuremath{\mathbb{R}}}
\newcommand{\Lr}[1]{\ensuremath{\mathcal{L}_r(#1)}}

\usepackage[backend=biber,maxnames=999,giveninits=true,style=trad-plain,sortcites=true,datamodel=mrnumber,isbn=false,url=false,doi=false]{biblatex}
\usepackage{csquotes}

\DeclareFieldFormat{title}{\myhref{\mkbibemph{#1}}}
\DeclareFieldFormat
  [article,inproceedings,unpublished]
  {title}{\myhref{\mkbibquote{#1\isdot}}}

\newcommand{\myhref}[1]{%
  \ifhyperref
    {\iffieldundef{doi}
       {\iffieldundef{url}
          {#1}
          {\href{\strfield{url}}{#1}}}
       {\href{http://dx.doi.org/\strfield{doi}}{#1}}}
    {#1}%
}

\DeclareFieldFormat{mrnumber}{%
  MR\addcolon\space
  \ifhyperref
    {\href{http://www.ams.org/mathscinet-getitem?mr=#1}{\nolinkurl{#1}}}
    {\nolinkurl{#1}}}

\renewbibmacro*{doi+eprint+url}{%
  \printfield{mrnumber}%
  \ifentrytype{online}
    {\newunit\newblock\usebibmacro{url+urldate}}
    {}%
}

\newcommand{\mathlib}{\gh{https://github.com/leanprover-community/mathlib4/tree/v4.30.0}}
\newcommand{\banlat}{\gh{https://github.com/davidmunozlahoz/banlat/tree/v0.1.0}}

\newcommand{\falg}{$f\!$-algebra}
\newcommand{\one}{\mathbb{1}}
\newcommand{\fn}[1]{\|#1\|_{\FBL[E]}}
\newcommand{\FBL}[1][E]{\operatorname{FBL}[#1]}
\DeclareMathOperator{\conv}{conv}
\DeclareMathOperator{\SolInt}{Sol\,Int}
\DeclareMathOperator{\id}{id}

\title{The Banach lattice Lean library}
\author{David Muñoz-Lahoz\footnote{Instituto de Ciencias
Matemáticas. Universidad Autónoma de Madrid.
\\\emph{Email:} \texttt{david.munnozl (at) uam (dot) es}}}
\date{\today}

\begin{document}

\maketitle

\begin{abstract}
    We present a Lean 4 library for the theory of Banach
    lattices. Its purpose is to support the systematic formalization of
    contemporary research in Banach lattices and related areas. As
    evidence of this, we describe three research-level formalizations
    built using the library. Writing the library at scale was made
    possible by the use of LLMs with careful human supervision and
    planning. Unlike autoformalization, this approach allows for
    an actual understanding of the code. This, in turn, led to new
    mathematical insights that are also discussed. Judging by the
    interest expressed by other researchers, we expect the
    library to become a communal effort in the near future. For this
    reason, we also describe several parts of the theory that could be added next.
\end{abstract}

\section{Introduction}

The Banach lattices Lean library \banlat\ is a
\href{https://lean-lang.org/}{Lean 4} project
that contains a body of elementary results in the theory of Banach
lattices. The goal of this library is to allow for
the systematic formalization of research-level results in Banach
lattices and nearby fields (such as stable phase
retrieval, where autoformalization is already being used
\cite{bertolini_etal2026}). In fact, in \cref{sec:use_cases} we will
present three examples of research-level results in Banach lattices
that have been formalized using the library. This was done by the
author in a relatively short span of time; thus we expect that, as the
library scales both in contributors and over time, we will be able to
systematically formalize a significant part of the research in these
fields.

\paragraph{``Semiautoformalization''}
This formalization at scale has been made possible by the significant
improvement of Large Language Models (LLMs).
Even though LLMs were used intensively in the process of writing this code,
we would not consider this an autoformalization project. Instead,
it could be regarded as an exercise in
\emph{semiautoformalization}: most of the code was written by
LLMs, but with close human steering and supervision. In the
levels of quality of formalization introduced by T.\ Tao in
\cite{tao2026}, this project would lie somewhere between levels 2
(Publication-level formalization) and 3 (Prototype-level
formalization).

This has several implications. On the one hand, the use of generative
tools allows for sufficient speed and scale to verify
research-level results in a reasonable amount of time and with reasonable effort.
The scientific implications of this are well-known: the amount
of literature being published is costly and difficult to check;
if formalization of papers is systematized, in order to check the
correctness of a paper we would only need to check the correctness of
the formalization. This may still not be entirely trivial, but it is
certainly easier. Notice that, in the process of checking these
formalized papers, having a trusted base of knowledge (a library) is
critical (for instance, in order to avoid checking the tree of
definitions back to mathlib every single time).

On the other hand, the careful human supervision of these generative
tools has several advantages. First, it allows for a careful
organization of the library. Second, we can correct certain approaches
proposed by the LLM that feel unnatural to humans, thus making the library more
transparent for future users and contributors. And third, since we
have an understanding of what is going on in our code, we can even
gain some mathematical insights in the process (see
\cref{sec:insights}).
In practice, doing this requires a certain expertise both in Banach
lattices and in Lean, careful planning of how to approach the
library, and making sure that we always ask the LLM to write definitions
and results that are close and within reach, so as not to let it make
design decisions. This keeps LLM usage efficient, another feature of
this approach that could become relevant if LLM token prices increase
in the future.

While working on this project, \cite{bei_etal2026} announced a library
in computational economics with a similar philosophy. We also
highlight \cite{armstrong2026}, which is used to formalize
\cite{armstrong_kuusi2025}. Very recently, members of the Lean FRO and
the Mathlib initiative announced Tau Ceti
\cite{taucetiproject2026}, a ``repository of formal mathematics,
directed by human-written roadmaps, implemented and maintained by AI
contributors, subject to adversarial review.''
We believe the philosophy behind this new project is closely aligned
with the one presented below.

\paragraph{Formalization vs ``Semiautoformalization'' vs
Autoformalization}

Of course, the previous benefits are amplified in manual
formalization: it is more readable, it is properly implemented, and
it provides more mathematical insights. This is, of course, what happens
with \mathlib, the mathematics Lean library
\cite{mathlib_community2020}. Then, what is the point of this
``semiautoformalization''? The answer is scale. It is a trade-off: we
are sacrificing code quality (especially inside the proofs, since
these are generally not checked) for scale. For
instance, as mentioned above, sometimes we had to redirect the LLM's
proposed approach, but we certainly also missed several occasions
to do so because of speed and scale. (At other times it was the LLM
that taught us how to approach certain results, see
\cref{sec:insights}.)

We do not feel that ``semiautoformalization'' is better or worse in
any way: it is just a different approach, with different goals in
mind. We believe that with this scale, and proper training of our
fellows, formal verification of new mathematical results can become
truly systematic in our field. How we are facing this community-level
change is addressed in more detail in \cite{de_dios_pont_etal2026}.

As a proof of concept, we present in \cref{sec:use_cases} the formalization of
three research-level results that the library has made possible: the
first is the construction of an object that has sparked considerable
interest in our community in recent years, the second is the solution
to an open problem recently published by the author, and the third is
a relatively simple counterexample answering an open question that is
part of an ongoing project.

We must clarify that, once all the necessary preliminaries have been
added to the library, we expect the formalization of papers to be an
actual autoformalization. For this process, there already exist tools
like LEAP \cite{kung_etal2026} and LeanMarathon \cite{zhang_etal2026}.
However, this kind of autoformalization, which assumes the necessary
preliminaries are already in the library, is significantly cheaper
than most autoformalization projects, and much easier to check because
of the shared body of results that the library provides.

\paragraph{Why Banach lattices?}

Banach lattices can be approached both from analysis, as Banach spaces
with some additional structure, and from order theory, as vector
lattices with a complete lattice norm. This second approach has proved
to be very apt for formalization, because it starts very close to
the contents already available in \mathlib. Thus we could get to the
elementary theory of vector and Banach lattices without much
preliminary material, progressing then to more analytic concepts
involving convergence, continuity, functionals, measures, compacta,
and more categorical constructions such as completions and free
objects, which play a central role in the modern theory.

Beyond the mathematical details, the field of Banach lattices has a cohesive and
supportive community that has been willing to join forces to
contribute to this project, and to learn the
new tools needed to do so. Several researchers in Banach
lattices have already expressed their interest in contributing. This
will make expanding the library in a variety of directions relatively
easy, leveraging the expertise of the different researchers.
A detailed account of how this process is unfolding can be found in
\cite{de_dios_pont_etal2026}.

\paragraph{Future steps}

Short-term goals of the library include its extension (see
\cref{sec:extension} for some candidate directions) and the
formalization of more research-level papers. In the near future, we
also expect to have our own systematic pipeline for formalizing papers using
the library. In the long term, we expect to distill from the library
the contents that could be welcome in \mathlib, rewrite them to have
the appropriate code quality, and contribute them. Since we are
building on top of the enormous communal effort that mathlib
represents, we feel that this could be an opportunity to give back.

Future improvements in technology (particularly, artificial
intelligence) could significantly increase the quality and scalability of
the project. We are not used to this in pure mathematics, where, unlike
in other sciences, progress has been essentially detached from
technological advances. However, projects like the one presented here
are highly dependent on technology. This means that procedures and
techniques could become obsolete as technology advances.

\paragraph{Notation and conventions}

No code appears in this paper. All pieces of code are provided as
external links to \href{https://github.com/}{GitHub repositories}.
This way, the code is always presented in its full context.
A link to a GitHub repository appears as
\gh{https://github.com/davidmunozlahoz/banlat/tree/v0.1.0}; a
folder inside the repository as
\gh{https://github.com/davidmunozlahoz/banlat/tree/v0.1.0/BanLat};
a file inside a folder as
\gh{https://github.com/davidmunozlahoz/banlat/blob/v0.1.0/BanLat/Basic.lean};
and a line inside a file as
\gh{https://github.com/davidmunozlahoz/banlat/blob/ef1bc06da46e90ed7934b35c11be07d7395bdaec/BanLat/Basic.lean\#L34}.
Often, instead of citing a particular line, we want to cite a
particular named Lean term (generally a theorem or definition). In this
case, the name instead of the line number will appear in the
reference, with the link pointing to the particular line where the
term is introduced. For instance,
\gh[abs_eq_zero_iff_zero]{https://github.com/davidmunozlahoz/banlat/blob/ef1bc06da46e90ed7934b35c11be07d7395bdaec/BanLat/Basic.lean\#L34}
points to the same line as
\gh{https://github.com/davidmunozlahoz/banlat/blob/ef1bc06da46e90ed7934b35c11be07d7395bdaec/BanLat/Basic.lean\#L34}.

Most of the time we cite the repositories \banlat\ and \mathlib.
These are libraries that change with time. Hence all references are to
the commit tagged as \verb|v0.1.0| in \banlat, and to the commit
tagged as \verb|v4.30.0| in \mathlib.

\section{Some use cases}\label{sec:use_cases}

In this section we present three research-level results that have been
formalized using the library. \Cref{sec:fbl} presents the
formalization of the free Banach lattice generated by a Banach space;
\cref{sec:wickstead} the formalization of a solution to an
open problem by A.\ W.\ Wickstead recently solved by the author; and
\cref{sec:lcs} the formalization of a counterexample to an
open question that is part of an ongoing project of the author
together with M.\ A.\ Taylor and P.\ Tradacete.

\subsection{The free Banach lattice generated by a Banach
space}\label{sec:fbl}

The free Banach lattice generated by a Banach space was first
constructed explicitly by A.\ Avilés, J.\ Rodríguez, and P.\ Tradacete
in their 2018 paper \cite{aviles_rodriguez_tradacete2018}. This object
has sparked a lot of research over the last decade, and has
been used to solve open problems as well as to better understand the
interaction between Banach spaces and Banach lattices (see for
instance
\cite{oikhberg_etal2022}).
For this reason, it seemed the best candidate to add a modern object
to the library on which a significant amount of recent research depends.

As the name indicates, the free Banach lattice generated by a Banach
space $E$ is a Banach lattice $\FBL $ together with a canonical
isometric embedding $\delta \colon E\to \FBL $ that satisfies the
following universal property: for every Banach lattice $X$ and every
bounded operator $T\colon E\to X$, there exists a unique lattice
homomorphism $\hat{T}\colon \FBL\to X$ with $\|\hat{T}\|=\|T\|$
satisfying $\hat{T}\delta =T$. As usual, if such an object exists,
then it is unique up to a unique lattice isometry. The major
accomplishment of \cite{aviles_rodriguez_tradacete2018} was to provide
an explicit description of $\FBL $. Next we describe this
construction, together with its formalization in \banlat. Throughout this
section, $E$ will denote a Banach space.

\begin{defn}
    <\gh[freeNormExt]{https://github.com/davidmunozlahoz/banlat/blob/ef1bc06da46e90ed7934b35c11be07d7395bdaec/BanLat/Free/FBL.lean\#L37}>
    Define the \emph{(extended) free norm} $\fn{{\cdot}}\colon
    C(B_{E^{*}})\to [0,\infty ]$ by
    \[
        \fn{f}=\sup \bigg\{\sum_{k=1}^{n}|f(x_k^{*})|\colon n \in \N,
            x_1^{*},\ldots ,x_n^{*}\in B_{E^{*}}, \sup_{x \in
            B_E}\sum_{k=1}^{n}|x_k^{*}(x)|\le 1\bigg\}
    \]
     for every $f \in C(B_{E^{*}})$.
\end{defn}

\begin{prop}
    <\gh[instBanachLatticeFunctionSpace]{https://github.com/davidmunozlahoz/banlat/blob/ef1bc06da46e90ed7934b35c11be07d7395bdaec/BanLat/Free/FBL.lean\#L558}>
    The space
    \[
        C_0=\{\, f \in C(B_{E^{*}}) : \fn{f}<\infty  \, \},
    \]
    with the order and linear structures inherited from
    $C(B_{E^{*}})$, and the restriction of $\fn{{\cdot }}$ as a norm,
    is a Banach lattice.
\end{prop}

For every $x \in E$, define $\delta _x \in C_0$ by $\delta
_x(x^{*})=x^{*}(x)$ for all $x^{*}\in B_{E^{*}}$
(\gh[dualEval]{https://github.com/davidmunozlahoz/banlat/blob/ef1bc06da46e90ed7934b35c11be07d7395bdaec/BanLat/Free/FBL.lean\#L561}).
Define $\FBL$ as the closed sublattice generated by $\{\, \delta _x :
x \in E \, \} $ in $(C_0,\fn{{\cdot }})$
(\gh[FBL]{https://github.com/davidmunozlahoz/banlat/blob/ef1bc06da46e90ed7934b35c11be07d7395bdaec/BanLat/Free/FBL.lean\#L592}).
Then $(\FBL, \fn{{\cdot }})$ is certainly a Banach lattice
(\gh[instBanachLatticeFBL]{https://github.com/davidmunozlahoz/banlat/blob/ef1bc06da46e90ed7934b35c11be07d7395bdaec/BanLat/Free/FBL.lean\#L663}).
Define also $\delta \colon E\to \FBL$ by $\delta (x)=\delta _x$ for
$x \in E$
(\gh[of]{https://github.com/davidmunozlahoz/banlat/blob/ef1bc06da46e90ed7934b35c11be07d7395bdaec/BanLat/Free/FBL.lean\#L671});
this map is a linear isometry
(\gh[ofLinearIsometry]{https://github.com/davidmunozlahoz/banlat/blob/ef1bc06da46e90ed7934b35c11be07d7395bdaec/BanLat/Free/FBL.lean\#L718}).
Then $\FBL$, together with $\delta $, is the free Banach lattice
generated by $E$:

\begin{thm}
    <\gh[exists_unique_lift]{https://github.com/davidmunozlahoz/banlat/blob/ef1bc06da46e90ed7934b35c11be07d7395bdaec/BanLat/Free/FBL.lean\#L1626}>
    Let $X$ be a Banach lattice, and let $T\colon E\to X$ be a bounded
    linear operator. Then there exists a unique lattice homomorphism
    $\hat{T}\colon \FBL\to X$ with $\|\hat{T}\|=\|T\|$ satisfying
    $\hat{T}\delta =T$.
\end{thm}

\subsection{Non-representable Banach lattice
algebras}\label{sec:wickstead}

In \cite{wickstead2017_two}, A.\ W.\ Wickstead formulated the following question. Let $X$ be
an order complete Banach lattice, and let $P\colon X\to X$ be a positive
projection for which there exists an $\alpha \ge 0$ satisfying:
\begin{enumerate}
\item $\alpha \id_X\le P$, and
\item whenever $T$ is a positive operator such that both $T\le
    \lambda \id_X$, for some $\lambda \ge 0$, and $T\le P$, then
    also $T\le \alpha \id_X$.
\end{enumerate}
Which values can $\alpha $ take? Wickstead conjectured that either
$\alpha =0$ or $\alpha =1/n$, for some $n \in \N$, and showed this for
finite-dimensional $X$. In \cite{munoz-lahoz2026_wickstead}, it was shown that this holds for
arbitrary $X$. In \gh{https://github.com/davidmunozlahoz/wickstead}
the main results of \cite{munoz-lahoz2026_wickstead} have been formalized using the Banach
lattices Lean library. More precisely, the general result is a
consequence of the $C(K)$-case, which can be proved without assuming
order completeness.

\begin{thm}
    <\gh[main_CK]{https://github.com/davidmunozlahoz/Wickstead/blob/d0ad6db6c5bfd596ca0d6ebc918d5fe3ffd3b34b/Wickstead/FinalAssembly.lean\#L90}>
Let $K$ be a non-empty compact Hausdorff space. Let $P\colon C(K)\to C(K)$
be a positive projection with $P\one_K=\one_K$. Suppose
there exists some $\alpha \ge 0$ satisfying:
\begin{enumerate}
    \item $\alpha \id_{C(K)}\le P$, and
    \item if $0\le M_f\le P$, then $\|f\|_\infty \le \alpha $.
\end{enumerate}
Then either $\alpha =0$ or $\alpha =1/n$ for some $n \in \N$.
\end{thm}

\begin{cor}
    <\gh[main]{https://github.com/davidmunozlahoz/Wickstead/blob/d0ad6db6c5bfd596ca0d6ebc918d5fe3ffd3b34b/Wickstead/CorollaryThreeOne.lean\#L1149}>
    \label{cor:main}
Let $X$ be a non-zero Dedekind complete vector lattice. Let $P\colon X\to
X$ be a positive projection with constant diagonal
$\alpha $. Then either $\alpha =0$ or $\alpha =1/n$, for some $n
\in \N$.
\end{cor}

This question in operator theory arose while studying the representation problem for
Banach lattice algebras. A \emph{Banach lattice algebra}
(\gh[BanachLatticeAlgebra]{https://github.com/davidmunozlahoz/Wickstead/blob/b614cd74d9850d812e5a98b715c36b74ec0989dc/Wickstead/PropositionFourThree.lean\#L44})
is a Banach lattice together with a Banach algebra structure for which
the product of positive elements is positive. It was an open problem
\cite[Question 5.3]{wickstead2017_questions} and \cite[Problem
2]{wickstead2017_open} whether every Banach lattice algebra is faithfully
representable (as a lattice and as an algebra) in the space of regular
operators on an order complete vector lattice. Wickstead showed
in \cite{wickstead2017_two} that, if \cref{cor:main} is true, then there exist
unital Banach lattice algebras that are not unitarily representable.
In \cite{munoz-lahoz2026_wickstead}, we were able to extend this approach to show that, in
fact, \cref{cor:main} implies a negative answer to the general representation
problem for Banach lattice algebras.

\begin{prop}
    <\gh[non_representable]{https://github.com/davidmunozlahoz/Wickstead/blob/b614cd74d9850d812e5a98b715c36b74ec0989dc/Wickstead/PropositionFourThree.lean\#L1044}>
Let $A$ be a Banach lattice algebra. Suppose there exist non-zero
positive elements $e,p \in A_+$ satisfying:
\begin{enumerate}
    \item $e^2=e$, $p^2=p$,
    \item $ep=pe=p$, and
    \item $p=\alpha e+x$, for some $x \in A$ with $x\wedge e=0$, and
        some $\alpha \in \R_+$.
\end{enumerate}
If $\alpha \not\in \{0\}\cup \{1/n:n \in \N\}$, then $A$ has no
faithful representation on $\Lr X$, for any order complete vector
lattice $X$.
\end{prop}

\subsection{\texorpdfstring{$\sigma $}{sigma}-Lebesgue locally convex-solid
topologies}\label{sec:lcs}

The following question appears in \cite[Open Problem 8.1]{aliprantis1974}:
is the completion of a $\sigma $-Lebesgue Hausdorff locally
convex-solid vector lattice also $\sigma $-Lebesgue?
The answer turns out to be negative, and a counterexample is formalized in
\gh{https://github.com/davidmunozlahoz/SigmaLebesgueCompletion}.

\begin{prop}
    <\gh[sigmaLebesgueTopology_not_preserved_by_completion]{https://github.com/davidmunozlahoz/SigmaLebesgueCompletion/blob/0778e0584ca78cc177910d17152f53bc0689e0f2/SigmaLebesgueCompletion/Counterexample.lean\#L316}>
    There exists a $\sigma $-Lebesgue Hausdorff locally convex-solid
    vector lattice whose completion is not $\sigma $-Lebesgue.
\end{prop}

\section{Some mathematical insights}\label{sec:insights}

This section collects some mathematical ideas that appeared in the
process of building the library and that we believe are relevant
enough to deserve a detailed explanation. Throughout this section, $n \in
\N$ denotes a natural number, and $X$ denotes a vector lattice.

\subsection{Vanishing lattice-linear expressions}

A classical result in the theory of vector lattices states that every
expression involving the linear and lattice operations (i.e., sums,
scalar multiplications, suprema, and infima) that holds in $\R$ must
also hold in every vector lattice. This result has important pragmatic
implications (it is relatively easy to check expressions in the reals,
because the order is total) and theoretical implications (for
instance, it is key in the concrete construction of the free vector
lattice).

When trying to formalize this result in Lean, we were expecting to do
it for Archimedean vector lattices using Kakutani's theorem. This is
the approach traditionally taken in Banach lattices monographs; see, for example,
\cite[Section 1.d]{lindenstrauss_tzafriri1979} and \cite[Theorem
2.1.10]{meyer-nieberg1991}.\footnote{This technique can be used to
prove something stronger, namely, positively homogeneous function
calculus, which is what the cited monographs end up proving.} We even
pointed the LLM toward this approach. However, the LLM took a different
route that surprised us, providing a proof that is elementary, does
not depend on the axiom of choice, and works for every vector lattice
(Archimedean or not).  This is remarkable, since in the more general
context of vector lattices, the result is proved using the
representation of a vector lattice inside the direct product of
totally ordered vector lattices, a technique that also relies on the
axiom of choice (see, for instance, \cite{labuschagne_vanalten2007}).

The result can also be seen as a consequence of the fact that the free
vector lattice can be realised as a vector lattice of real-valued
functions. This was proved by R.\ D.\ Bleier \cite{bleier1973} using
convex analysis, an approach that probably inspired the proof presented
here. Bleier's argument is not entirely elementary, since his proof of
\cite[Lemma 2.1]{bleier1973} again relies on the representation of a
vector lattice as a subdirect product of totally ordered vector
lattices; however, as observed by M.\ de Jeu
\cite[Lemma 2.1]{dejeu2021}, the use of representation theorems can be
avoided, making
the argument completely elementary.\footnote{Bleier's result also
follows from a more general theorem of M.\ de Jeu and X.\ Jiang
\cite[Theorem 4.5]{dejeu_jiang2026}, which canonically realises the
free vector lattice over a pre-ordered vector space as a vector
sublattice of a function lattice.} Nevertheless, the argument
used in the library avoids free objects altogether, which keeps it
elementary and, more importantly, makes it easier to formalize. This is particularly relevant
because Bleier's approach requires the existence of the free vector
lattice, a fact that comes from universal algebra and whose
formalization would require developing a completely different area in
Lean.

For these reasons, we believe it is worth explaining the details of
this new proof, together with its formalization. Even though the idea of the
proof is very clear, writing down all the details is cumbersome
because doing so requires manipulating formal expressions, which are
defined recursively. However, since Lean is by design very good at
dealing with recursive definitions, this is one of those rare
instances in which an argument is more elegant and direct in Lean than
in written natural language!

To state the theorem properly, we first need the formal definition of
``expressions involving the linear and lattice operations.''

\begin{defn}<\gh[LLexpr]{https://github.com/davidmunozlahoz/banlat/blob/ef1bc06da46e90ed7934b35c11be07d7395bdaec/BanLat/LLexpr.lean\#L22}>
    A \emph{lattice-linear expression $\Phi $ in the
    formal variables $t_1,\ldots ,t_n$} is a formal expression
    constructed inductively in the following way:
    \begin{enumerate}
        \item $0$ is a lattice-linear expression;
        \item $t_1$, $t_2$, \ldots, $t_n$ are lattice-linear expressions;
        \item if $\Phi _1$ and $\Phi _2$ are lattice-linear
            expressions in $t_1,\ldots ,t_n$, then $\Phi _1+\Phi _2$,
            $\lambda \Phi _1$ (for $\lambda \in \R$), $\Phi _1 \vee \Phi _2$, and $\Phi _1
            \wedge \Phi _2$ are lattice-linear expressions in
            $t_1,\ldots ,t_n$.
    \end{enumerate}
\end{defn}

\begin{rem}
    To abbreviate, we denote the fact that $\Phi $ is a lattice-linear
    expression in the formal variables $t_1,\ldots, t_n$ by $\Phi
    [t_1,\ldots ,t_n]$. Note that not all the variables need appear
    explicitly in $\Phi $; for example, $\Phi [t_1,\ldots ,t_n]=t_1$
    is a lattice-linear expression in $t_1,\ldots ,t_n$. In
    particular, every lattice-linear expression in $t_1,\ldots ,t_n$
    is also a lattice-linear expression in $t_1,\ldots ,t_m$ for all
    $m\ge n$.
\end{rem}

It should be intuitively clear that every lattice-linear expression in
$n$ variables can be ``evaluated'' at $n$ arbitrary elements of a
vector lattice in the ``obvious way.'' This notion is defined
recursively as follows.

\begin{defn}
    <\gh[eval]{https://github.com/davidmunozlahoz/banlat/blob/ef1bc06da46e90ed7934b35c11be07d7395bdaec/BanLat/LLexpr.lean\#L38}>
    Let $\Phi [t_1,\ldots ,t_n]$ be a
    lattice-linear expression, and let
    $x_1,\ldots ,x_n \in X$. We define the \emph{evaluation} of $\Phi$ at
    $x_1,\ldots ,x_n$ as the element $\Phi (x_1,\ldots ,x_n) \in X$
    defined recursively by
    \begin{enumerate}
        \item if $\Phi =0$, $\Phi (x_1,\ldots ,x_n)=0$.
        \item if $\Phi =t_i$, $\Phi (x_1,\ldots ,x_n)=x_i$.
        \item if $\Phi =\Phi _1+\Phi _2$, for some lattice-linear
            expressions $\Phi _1$ and $\Phi _2$, $\Phi (x_1,\ldots
            ,x_n)=\Phi _1(x_1,\ldots ,x_n)+\Phi _2(x_1,\ldots ,x_n)$.
        \item if $\Phi =\lambda \Phi _1$, for some lattice-linear
            expression $\Phi _1$ and some $\lambda \in \R$, $\Phi (x_1,\ldots
            ,x_n)=\lambda \Phi _1(x_1,\ldots ,x_n)$.
        \item if $\Phi =\Phi _1\vee \Phi _2$, for some lattice-linear
            expressions $\Phi _1$ and $\Phi _2$, $\Phi (x_1,\ldots
            ,x_n)=\Phi _1(x_1,\ldots ,x_n)\vee \Phi _2(x_1,\ldots ,x_n)$.
        \item if $\Phi =\Phi _1\wedge \Phi _2$, for some lattice-linear
            expressions $\Phi _1$ and $\Phi _2$, $\Phi (x_1,\ldots
            ,x_n)=\Phi _1(x_1,\ldots ,x_n)\wedge \Phi _2(x_1,\ldots ,x_n)$.
    \end{enumerate}
\end{defn}

The above is an example of a definition that is shorter to write
in Lean than in \LaTeX, all thanks to Lean's pattern matching.

\begin{defn}
    <\gh[Vanishes]{https://github.com/davidmunozlahoz/banlat/blob/ef1bc06da46e90ed7934b35c11be07d7395bdaec/BanLat/LLexpr.lean\#L127}>
    We say that a lattice-linear expression $\Phi[t_1,\ldots ,t_n]$
    \emph{vanishes on $X$} if $\Phi (x_1,\ldots ,x_n)=0$ for all
    $x_1,\ldots ,x_n \in X$.
\end{defn}

We can now state the theorem.

\begin{thm}
    <\gh[vanishes_of_vanishes_real]{https://github.com/davidmunozlahoz/banlat/blob/ef1bc06da46e90ed7934b35c11be07d7395bdaec/BanLat/LLexpr.lean\#L476}>
    \label{thm:llexpr_zero}
    A lattice-linear expression that vanishes on $\R$ must also vanish
    on every vector lattice.
\end{thm}

The idea of the proof is the following. Consider two arbitrary lattice-linear
expressions $\Phi[t_1,\ldots ,t_n]$ and $\Psi[t_1,\ldots ,t_n]$ that
are suprema of linear expressions, that is, that can be written as
\[
    \Phi[t_1,\ldots ,t_n] =\bigvee
    _{i=1}^{m}\sum_{k=1}^{n}a_{ik}t_k\quad\text{and}\quad\Psi [t_1,\ldots
    ,t_n]=\bigvee_{j=1} ^{l} \sum_{k=1}^{n}b_{ik}t_k,
\]
for some $m,l \in \N$ and $a_{ik},b_{ik}\in \R$. First we are going to
show, using elementary computations, that if $\Phi (\lambda _1,\ldots
,\lambda _n)=\Psi (\lambda _1,\ldots ,\lambda _n)$ for all $\lambda
_1,\ldots ,\lambda _n \in \R$, then $\Phi (x_1,\ldots ,x_n)=\Psi
(x_1,\ldots ,x_n)$ for all $x_1,\ldots ,x_n$ in an arbitrary vector
lattice $X$.

Next we are going to prove that, if now
$\Phi [t_1,\ldots ,t_n]$ is an arbitrary lattice-linear expression,
then there exists an ``equivalent''
lattice-linear expression $\tilde \Phi [t_1,\ldots ,t_n]$ (equivalent
in the sense that $\Phi (x_1,\ldots ,x_n)=\tilde \Phi (x_1,\ldots ,x_n)$
for every $x_1,\ldots ,x_n \in X$) of the
form
\[
    \Phi[t_1,\ldots ,t_n] =\bigvee
    _{i=1}^{m}\sum_{k=1}^{n}a_{ik}t_k-\bigvee_{j=1} ^{l}
    \sum_{k=1}^{n}b_{jk}t_k.
\]
This is the step that requires working recursively on the definition
of $\Phi $, and that is very clean in Lean but technical in
written math. Once this is proved, it is easy to conclude: if
$\Phi $ vanishes on the reals, so does $\tilde \Phi $, which by the
previous paragraph implies that the two suprema appearing in the
expression of $\tilde \Phi $ must evaluate to the same element on every
vector lattice. Thus $\tilde \Phi $ vanishes on every vector lattice,
and so does $\Phi $.

We now turn to the actual details of the proof.
We start with two elementary lemmas. For $\bm a_1,\ldots ,\bm a_m \in
\R^{n}$, we denote by $\conv(\bm a_1,\ldots ,\bm a_m)$ their convex
hull.

\begin{lem}
    <\gh[mem_convexHull_of_forall_le_sup]{https://github.com/davidmunozlahoz/banlat/blob/ef1bc06da46e90ed7934b35c11be07d7395bdaec/BanLat/LLexpr.lean\#L263}>
    \label{lem:inconvhull}
    Let $m \in \N$, let $\bm a_{i}=(a_{i1},\ldots ,a_{i n}) \in \R^{n}$, for
    $i=1,\ldots ,m$, and let $\bm p =(p_1,\ldots ,p_n)\in \R^{n}$.
    If, for every $(\lambda _1,\ldots ,\lambda _n)\in \R^{n}$,
    \[
    \sum_{j=1}^{n}\lambda _jp_j \le \bigvee _{i=1}^{m} \sum_{j=1}^{n}
    \lambda _j a_{ij},
    \]
    then $\bm p \in \conv(\bm a_1, \ldots ,\bm a_m)$.
\end{lem}
\begin{proof}
    Since $\conv(\bm a_1, \ldots ,\bm a_m)$ is compact, the result is
    an immediate consequence of the hyperplane separation
    theorem (see, for instance, \cite[Section
    2.5.1]{boyd_vandenberghe2004}; or
    \gh[geometric_hahn_banach_point_closed]{https://github.com/leanprover-community/mathlib4/blob/c5ea00351c28e24afc9f0f84379aa41082b1188f/Mathlib/Analysis/LocallyConvex/Separation.lean\#L223}
    in \mathlib).
\end{proof}

\begin{lem}
    <\gh[linearCombination_le_eval_of_mem_convexHull]{https://github.com/davidmunozlahoz/banlat/blob/ef1bc06da46e90ed7934b35c11be07d7395bdaec/BanLat/LLexpr.lean\#L300}>
    \label{lem:leconvhull}
    Let $\Phi [t_1,\ldots ,t_n]$ be a lattice-linear
    expression of the form $\Phi [t_1,\ldots ,t_n]=\bigvee_{i=1}
    ^{m}\sum_{j=1}^{n}a_{ij}t_j$, for some $m \in \N$ and $a_{ij}\in
    \R$. Set $\bm a_i=(a_{i1},\ldots
    ,a_{i n})$, for $i=1,\ldots ,m$.
    If $\bm p=(p_1,\ldots ,p_n) \in \conv(\bm a_1,\ldots ,\bm a_m)$,
    then
    \[
    \sum_{i=1}^{n}p_i x_i \le \Phi (x_1,\ldots ,x_n),
    \]
    for every $x_1,\ldots ,x_n \in X$.
\end{lem}
\begin{proof}
    Since $\bm p \in \conv(\bm a_1,\ldots ,\bm a_m)$, there exist $\mu _1,\ldots ,\mu _m
    \ge 0$ with $\sum_{i=1}^{m}\mu _i=1$ such that
    \[
    \bm p=\sum_{i=1}^{m}\mu _i \bm a_i.
    \]
    To conclude, simply compute
    \begin{align*}
        \sum_{j=1}^{n}p_j x_j &= \sum_{j=1}^{n} \bigg( \sum_{i=1}^{m}
        \mu _i a_{ij} \bigg) x_j\\
                              &= \sum_{i=1}^{m} \mu _i
                              \sum_{j=1}^{n} a_{ij} x_j\\
                              &\le \sum_{i=1}^{m} \mu _i
                              \bigvee_{i=1} ^{m} \sum_{j=1}^{n}
                              a_{ij}x_j\\
                              &= \Phi (x_1,\ldots ,x_n).\qedhere
    \end{align*}
\end{proof}

Now we can prove the theorem for lattice-linear expressions that are
suprema of linear expressions.

\begin{lem}
    <\gh[eval_eq_of_forall_real_eq]{https://github.com/davidmunozlahoz/banlat/blob/ef1bc06da46e90ed7934b35c11be07d7395bdaec/BanLat/LLexpr.lean\#L336}>
    \label{lem:suplin}
    Let $\Phi [t_1,\ldots ,t_n]$ and $\Psi [t_1,\ldots
    ,t_n]$ be lattice-linear expressions of the form
    \[
        \Phi[t_1,\ldots ,t_n] =\bigvee
        _{i=1}^{m}\sum_{k=1}^{n}a_{ik}t_k\quad\text{and}\quad\Psi [t_1,\ldots
        ,t_n]=\bigvee_{j=1} ^{l} \sum_{k=1}^{n}b_{jk}t_k,
    \]
    for some $m,l \in \N$ and $a_{ik},b_{ik}\in \R$. Suppose that
    $\Phi (\lambda _1,\ldots ,\lambda _n)=\Psi (\lambda _1,\ldots
    ,\lambda _n)$ for every $\lambda _1,\ldots ,\lambda _n \in \R$. Then
    $\Phi (x_1,\ldots ,x_n)=\Psi (x_1,\ldots ,x_n)$ for every
    $x_1,\ldots ,x_n \in X$.
\end{lem}
\begin{proof}
    Let $x_1,\ldots ,x_n \in X$.
    Since the roles of $\Phi $ and $\Psi $ are interchangeable, it
    suffices to prove that
    \[
    \Phi (x_1,\ldots ,x_n)=\bigvee_{i=1}
    ^{m}\sum_{k=1}^{n}a_{ik}x_k\le \Psi (x_1,\ldots ,x_n).
    \]
    This, in turn, is the same as showing that
    \begin{equation}\label{eq:lin_leq_ll}
       \sum_{k=1}^{n}a_{ik}x_k\le \Psi (x_1,\ldots ,x_n)
    \end{equation}
    holds for each $i=1,\ldots ,m$. Fix $i \in \{1,\ldots ,m\}$. By assumption, for all $(\lambda
    _1,\ldots ,\lambda _n)\in \R^{n}$,
    \[
       \sum_{k=1}^{n}a_{ik}\lambda _k\le \Psi (\lambda _1,\ldots
       ,\lambda _n)=\bigvee_{j=1} ^{l} \sum_{k=1}^{n}b_{jk}\lambda _k.
    \]
    According to \cref{lem:inconvhull}, this implies that $\bm a_i \in
    \conv(\bm b_1,\ldots ,\bm b_l)$, where $\bm a_i = (a_{i
    1},\ldots ,a_{i n})$ and $\bm b_j=(b_{j 1},\ldots ,b_{j n})$ for
    $j=1,\ldots ,l$. Now \eqref{eq:lin_leq_ll} follows immediately
    from \cref{lem:leconvhull}.
\end{proof}

\begin{defn}
    Let $\Phi [t_1,\ldots ,t_n]$ and $\Psi [t_1,\ldots
    ,t_n]$ be lattice-linear expressions. We say that $\Phi $ and
    $\Psi $ are \emph{equivalent} if $\Phi (x_1,\ldots ,x_n)=\Psi
    (x_1,\ldots ,x_n)$ for every $x_1,\ldots ,x_n$ in a vector lattice
    $X$.
\end{defn}

\begin{defn}
    <\gh[NormalForm]{https://github.com/davidmunozlahoz/banlat/blob/ef1bc06da46e90ed7934b35c11be07d7395bdaec/BanLat/LLexpr.lean\#L347}>
    We say that a lattice-linear expression $\Phi
    [t_1,\ldots ,t_n]$ is in \emph{normal form} if it is of the form
    \[
        \Phi[t_1,\ldots ,t_n] =\bigvee
        _{i=1}^{m}\sum_{k=1}^{n}a_{ik}t_k-\bigvee_{j=1} ^{l}
        \sum_{k=1}^{n}b_{jk}t_k.
    \]
    for some $m, l \in \N$, and some $a_{ik}, b_{jk} \in \R$.
\end{defn}

The following is the last ingredient needed to complete the proof of
\cref{thm:llexpr_zero}. The proof of the following fact in Lean is
more elegant, and even shorter, than the one provided here. Also, the
computational nature of the proof makes it prone to typos, whereas the
formulas that appear in the Lean code are certainly correct.

\begin{thm}
    <\gh[normalize_eval]{https://github.com/davidmunozlahoz/banlat/blob/ef1bc06da46e90ed7934b35c11be07d7395bdaec/BanLat/LLexpr.lean\#L469}>
    \label{thm:normalize}
    For every lattice-linear expression $\Phi [t_1,\ldots ,t_n]$ there
    exists a lattice-linear expression $\tilde \Phi [t_1,\ldots ,t_n]$
    in normal form that is equivalent to $\Phi $.
\end{thm}
\begin{proof}
    We prove it by induction on the definition
    of $\Phi $. If $\Phi =0$ or $\Phi =t_i$, for some $i$, then $\Phi
    $ is already in normal form and there is nothing to do. This
    proves the base case. Now suppose that there exist lattice-linear
    expressions $\Phi _1$ and $\Phi _2$
    such that precisely one of the following holds:
    \begin{enumerate}
        \item $\Phi =\Phi _1+\Phi _2$.
        \item $\Phi =\lambda \Phi _1$, for some $\lambda \in \R$.
        \item $\Phi =\Phi _1\vee \Phi _2$,
        \item $\Phi =\Phi _1\wedge  \Phi _2$.
    \end{enumerate}
    Use the induction hypothesis to obtain lattice-linear expressions
    $\tilde \Phi _1$ and $\tilde \Phi _2$ in normal form equivalent to
    $\Phi _1$ and $\Phi _2$, respectively. We can write, explicitly,
    \[
        \tilde\Phi_1[t_1,\ldots ,t_{n_1}] =\bigvee
        _{i=1}^{m_1}\sum_{k=1}^{n_1}a_{ik}t_k-\bigvee_{j=1} ^{l_1}
        \sum_{k=1}^{n_1}b_{jk}t_k,
    \]
    and
    \[
        \tilde\Phi_2[t_1,\ldots ,t_{n_2}] =\bigvee
        _{i=1}^{m_2}\sum_{k=1}^{n_2}c_{ik}t_k-\bigvee_{j=1} ^{l_2}
        \sum_{k=1}^{n_2}d_{jk}t_k,
    \]
    for some $n_1,n_2 \in \{1,\ldots ,n\}$, $m_1,m_2,l_1,l_2 \in \N$,
    and $a_{ik}, b_{jk}, c_{ik}, d_{jk} \in \R$.

    Below we give an explicit expression for $\tilde \Phi $ depending
    on how $\Phi $ is constructed. In each of the cases, it is clear
    that $\tilde \Phi $ is in normal form, and it is immediate to
    check that it is equivalent to $\Phi $ using the basic properties
    of vector lattices. Let $m=\max\{m_1,m_2\}$ and
    $l=\max\{l_1,l_2\}$. By convention, the coefficients that are not
    defined are set equal to $0$.
    \begin{enumerate}
        \item If $\Phi =\Phi _1+\Phi _2$, then
            \[
                \tilde \Phi [t_1,\ldots ,t_n]=\bigvee
        _{i=1}^{m}\sum_{k=1}^{n}(a_{ik}+c_{ik})t_k-\bigvee_{j=1} ^{l}
        \sum_{k=1}^{n}(b_{jk}+d_{jk})t_k.
            \]
        \item If $\Phi =\lambda \Phi _1$, with $\lambda \ge 0$, then
            \[
                \tilde\Phi[t_1,\ldots ,t_n] =\bigvee
                _{i=1}^{m_1}\sum_{k=1}^{n_1}(\lambda a_{ik})t_k-\bigvee_{j=1} ^{l_1}
                \sum_{k=1}^{n_1}(\lambda b_{jk})t_k,
            \]
            whereas if $\lambda <0$, then
            \[
                \tilde\Phi[t_1,\ldots ,t_n] =\bigvee
                _{i=1}^{l_1}\sum_{k=1}^{n_1}(-\lambda b_{ik})t_k-\bigvee_{j=1} ^{m_1}
                \sum_{k=1}^{n_1}(-\lambda a_{jk})t_k.
            \]
        \item If $\Phi =\Phi _1 \vee \Phi _2$, then
            \begin{align*}
                \tilde \Phi [t_1,\ldots ,t_n]=&\bigg(\bigvee_{i=1}
                ^{\max\{m,l\}}
            \sum_{k=1}^{n}(a_{ik}+d_{ik})t_k\bigg)\vee \bigg(
        \bigvee_{i=1} ^{\max\{m,l\}}\sum_{k=1}^{n}(c_{ik}+b_{ik})t_k
    \bigg)\\  &- \bigvee_{j=1} ^{l}\sum_{k=1}^{n}(b_{jk}+d_{jk})t_k.
            \end{align*}
        \item If $\Phi =\Phi _1 \wedge \Phi _2$, then $\Phi $ is
            equivalent to $- ((-\Phi _1)\vee (-\Phi _2))$, and we can
            iterate (ii) and (iii) to obtain the desired
            lattice-linear expression.
    \end{enumerate}
\end{proof}

\begin{proof}[Proof of \cref{thm:llexpr_zero}.]
    Let $\Phi [t_1,\ldots ,t_n]$ be a lattice-linear expression that
    vanishes on $\R$. By \cref{thm:normalize}, there exists a
    lattice-linear expression $\tilde \Phi $ in normal form that is
    equivalent to $\Phi $. Thus it suffices to show that $\tilde \Phi
    (x_1,\ldots ,x_n)=0 $ for every $x_1,\ldots ,x_n$ in a vector
    lattice $X$. Write
    \[
        \tilde \Phi[t_1,\ldots ,t_n] =\bigvee
        _{i=1}^{m}\sum_{k=1}^{n}a_{ik}t_k-\bigvee_{j=1} ^{l}
        \sum_{k=1}^{n}b_{jk}t_k
    \]
    for some $m, l \in \N$, and some $a_{ik}, b_{jk} \in \R$. By
    assumption,
    \[
    \bigvee _{i=1}^{m}\sum_{k=1}^{n}a_{ik}\lambda _k=\bigvee_{j=1} ^{l}
    \sum_{k=1}^{n}b_{jk}\lambda _k\quad\text{for every }\lambda
    _1,\ldots ,\lambda _n \in \R.
    \]
    It follows from \cref{lem:suplin} that
    \[
    \bigvee _{i=1}^{m}\sum_{k=1}^{n}a_{ik}x _k=\bigvee_{j=1} ^{l}
    \sum_{k=1}^{n}b_{jk}x _k.
    \]
    Therefore $\Phi (x_1,\ldots ,x_n)=0$.
\end{proof}

This result was then used to implement
\gh[llarith]{https://github.com/davidmunozlahoz/banlat/blob/ef1bc06da46e90ed7934b35c11be07d7395bdaec/BanLat/Tactic/LLexpr.lean\#L474},
a tactic that closes non-strict inequalities in vector lattices by
checking that the inequality holds in $\R$. This in turn is done using
the
\href{https://leanprover-community.github.io/mathlib4_docs/Mathlib/Tactic/Linarith/Frontend.html}{linarith}
tactic. The tactic also supports assumptions of the type $0\le x$.

We believe that this tactic could be useful in the process of
autoformalizing papers in order to manipulate lattice-linear
identities quickly and conveniently. However, for performance reasons,
we try to avoid its use in the library files that are high in the
import hierarchy. This tactic was implemented using heavy LLM
assistance, and it is far from optimal; we could certainly
use the advice from more experienced metaprogrammers to improve it.

\subsection{The Archimedean property}

In \mathlib, the
Archimedean property is defined for ordered commutative monoids as
follows.

\begin{defn}\label{def:arch1}
    \gh[Archimedean]{https://github.com/leanprover-community/mathlib4/blob/c5ea00351c28e24afc9f0f84379aa41082b1188f/Mathlib/Algebra/Order/Archimedean/Defs.lean\#L32}
    An ordered commutative monoid $R$ is said to be \emph{Archimedean}
    if for every two elements $x,y \in R$ such that $0<y$, there
    exists a natural number $n \in \N$ such that $x\le ny$.
\end{defn}

This definition, although standard in the theory
of linearly ordered groups, is stronger than the one used in the
theory of lattice-ordered groups and vector lattices. For this reason,
in \banlat, we had to introduce the following alternative definition.

\begin{defn}
    <\gh[IsVLArchimedean]{https://github.com/davidmunozlahoz/banlat/blob/ef1bc06da46e90ed7934b35c11be07d7395bdaec/BanLat/Basic.lean\#L366}>
    \label{def:arch2}
    A lattice-ordered group $G$ is said to be \emph{Archimedean} if
    whenever $x,y \in G$ are such that
    $ny\leq x$ for all $n \in \N$, then $y\le 0$.
\end{defn}

Definitions~\ref{def:arch1} and~\ref{def:arch2} are equivalent in
linearly ordered groups. However, \cref{def:arch1} is in general
stronger than \cref{def:arch2}: $\R^2$, with the order defined
pointwise, satisfies \cref{def:arch2} but not \cref{def:arch1}.

This is not the only notational subtlety that arose when building the
library. For instance, what in ordered vector spaces and vector
lattices is usually called order completeness or Dedekind
completeness, in lattice theory (and, therefore, in mathlib) is known
as conditionally completeness (\gh[ConditionallyCompleteLattice]{https://github.com/leanprover-community/mathlib4/blob/c5ea00351c28e24afc9f0f84379aa41082b1188f/Mathlib/Order/ConditionallyCompleteLattice/Defs.lean\#L46}).

\subsection{Locally convex-solid topologies}

A subset $S\subseteq X$ is \emph{solid} if $[-|x|,|x|]\subseteq S$
holds for every $x \in S$. A linear topology on $X$ is \emph{locally
solid} if $0$ has a neighbourhood basis of solid sets; $X$ equipped
with a locally solid topology is called a \emph{locally solid vector
lattice}
(\gh[IsLocallySolidVectorLattice]{https://github.com/davidmunozlahoz/banlat/blob/ef1bc06da46e90ed7934b35c11be07d7395bdaec/BanLat/LocallySolid/Basic.lean\#L16}).
Locally solid topologies abstract many of the order-topological
properties of Banach lattices. Recall that, in a vector space, a
linear topology is locally convex if $0$ has a neighbourhood basis of
convex sets. A particularly relevant family of locally solid
topologies are the ones that are also locally convex.

\begin{defn}<\gh[IsLocallyConvexSolidVectorLattice]{https://github.com/davidmunozlahoz/banlat/blob/ef1bc06da46e90ed7934b35c11be07d7395bdaec/BanLat/LocallySolid/LocallyConvexSolid.lean\#L16}>
    A linear topology on a vector lattice that is at the same time
    locally solid and locally convex is called a \emph{locally
    convex-solid topology}. A vector lattice together with a locally
    convex-solid topology is called a \emph{locally convex-solid
    vector lattice}.
\end{defn}

This definition is, almost verbatim, the one given in the monograph
    \cite[Page 59]{aliprantis_burkinshaw2003}. In the same book, and without any further
comment on the definition, it is assumed that locally convex-solid
topologies admit a neighbourhood basis at the origin formed by sets
that are both convex and solid. But this is not strictly what the
definition says: the definition establishes that there exist both a
neighbourhood basis of convex sets and a neighbourhood basis of solid
sets. That there exists a basis with both properties at the same time
is reasonably easy, but not entirely obvious.

This subtlety arose when trying to formalize the notion of locally
convex-solid topology in Lean. Since this is not addressed in
\cite{aliprantis_burkinshaw2003}, the main reference monograph on the topic, it seems
relevant to provide the mathematical details next.
These ideas can be found in much greater generality in \cite[Section
2]{bilokoytov2025}.

\begin{defn}<\gh[solidInterior]{https://github.com/davidmunozlahoz/banlat/blob/ef1bc06da46e90ed7934b35c11be07d7395bdaec/BanLat/LocallySolid/LocallyConvexSolid.lean\#L24}>
    Given $S\subseteq X$, we define its \emph{solid interior} as the subset
    \[
        \SolInt(S)=\{\, x \in S : [-|x|,|x|]\subseteq S \, \}.
    \]
\end{defn}

Clearly, the solid interior of $S$ is the largest solid set contained
in $S$. The following is the key fact we need about the solid
interior.

\begin{prop}<\gh[convex_solidInterior]{https://github.com/davidmunozlahoz/banlat/blob/ef1bc06da46e90ed7934b35c11be07d7395bdaec/BanLat/LocallySolid/LocallyConvexSolid.lean\#L35}>
    \label{prop:solintconv}
    In a vector lattice, the solid interior of a convex set is convex.
\end{prop}
\begin{proof}
    Let $C\subseteq X$ be a convex set. Let $x,y \in \SolInt(C)$ and
    let $\lambda ,\mu \in \R_+$ be such that $\lambda +\mu =1$. We
    want to show that $[-|\lambda x+\mu y|,|\lambda x+\mu y|]\subseteq
    C$. If $\lambda =0$ or $\mu =0$, the result is immediate, so we
    may assume that $\lambda ,\mu >0$. Let $z \in X$ be such that $|z| \le |\lambda x+\mu y|$. Then
    $|z|\le \lambda |x| + \mu |y|$, and by the Riesz decomposition
    theorem there exist $w_1,w_2 \in X$
    with $|w_1|\le \lambda |x|$ and $|w_2|\le \mu |y|$ such that
    $z=w_1+w_2$. Since $x,y \in \SolInt(C)$, which of course is a
    solid set, it follows that $\lambda ^{-1}w_1,\mu ^{-1}w_2 \in \SolInt(C)\subseteq C$.
    Convexity of $C$ then implies
    \[
    z= \lambda (\lambda ^{-1}w_1)+\mu (\mu ^{-1}w_2) \in C.\qedhere
    \]
\end{proof}

\begin{prop}
    <\gh[hasBasis_convex_solid]{https://github.com/davidmunozlahoz/banlat/blob/ef1bc06da46e90ed7934b35c11be07d7395bdaec/BanLat/LocallySolid/LocallyConvexSolid.lean\#L88}>
    Let $\tau $ be a locally convex-solid topology on $X$. Then $0$
    has a neighbourhood basis formed by sets that are both convex and
    solid.
\end{prop}
\begin{proof}
    Let $U$ be a neighbourhood of $0$. By assumption, there exist
    neighbourhoods $S'\subseteq C\subseteq S\subseteq U$ of $0$ such
    that $S,S'$ are solid and $C$ is convex. Then
    $\SolInt(C)\subseteq U$ is solid, convex by
    \cref{prop:solintconv}, and it is a neighbourhood of $0$ because it
    contains $S'$.
\end{proof}

\section{Overview of the library}\label{sec:overview}

The following is an overview of what is currently available in the
library.

\begin{itemize}
  \item \gh{https://github.com/davidmunozlahoz/banlat/blob/v0.1.0/BanLat.lean}: General import file for the whole library.

  \item \gh{https://github.com/davidmunozlahoz/banlat/blob/v0.1.0/BanLat/Basic.lean}: Develops the basic definitions and
      identities of lattice-ordered groups and real vector lattices.
      Defines the Archimedean property and proves equivalent
      characterizations.

  \item \gh{https://github.com/davidmunozlahoz/banlat/blob/v0.1.0/BanLat/Normed.lean}: Defines normed vector lattices
      and Banach lattices. It proves continuity of lattice operations,
      Lipschitz continuity of the modulus map, closedness of the
      positive cone and order intervals, boundedness of order-bounded
      sets, and monotone convergence facts.

  \item \gh{https://github.com/davidmunozlahoz/banlat/blob/v0.1.0/BanLat/Disjoint.lean}: Introduces the notion of
      disjoint vectors in a vector lattice, and proves basic facts
      related to disjointness.

  \item \gh{https://github.com/davidmunozlahoz/banlat/blob/v0.1.0/BanLat/RieszDec.lean}: Proves the Riesz decomposition
      theorem and closely related results.

  \item \gh{https://github.com/davidmunozlahoz/banlat/blob/v0.1.0/BanLat/OrderComplete.lean}: Defines order and $\sigma
      $-order completeness, and proves basic characterizations (e.g.,
      it suffices to check completeness for increasing positive nets).

  \item \gh{https://github.com/davidmunozlahoz/banlat/blob/v0.1.0/BanLat/OrderDense.lean}: Defines order dense subsets,
      and proves some basic properties.

  \item \gh{https://github.com/davidmunozlahoz/banlat/blob/v0.1.0/BanLat/OrderUnit.lean}: Defines weak and strong order units.

  \item \gh{https://github.com/davidmunozlahoz/banlat/blob/v0.1.0/BanLat/Pi.lean}: Introduces a natural vector lattice
      structure in the arbitrary product of vector lattices. It also
      introduces $p$-sums, $1\le p\le \infty $, of normed and Banach
      lattices. Records that the $\infty $-sum of AM-spaces is an
      AM-space, and the $1$-sum of AL-spaces is an AL-space.

  \item \gh{https://github.com/davidmunozlahoz/banlat/blob/v0.1.0/BanLat/Atom.lean}: Defines atoms in vector lattices.
      Defines the continuous and atomic part of a vector lattice.
      Shows that every infinite-dimensional Archimedean vector
      lattice contains a sequence of pairwise disjoint non-zero
      vectors. Shows that an Archimedean vector lattice is finite-dimensional
      if and only if it is isomorphic to $\R^{n}$.

  \item \gh{https://github.com/davidmunozlahoz/banlat/blob/v0.1.0/BanLat/QuasiInteriorPoint.lean}: Defines
      quasi-interior points. Proves that every separable Banach
      lattice has a quasi-interior point.

  \item \gh{https://github.com/davidmunozlahoz/banlat/blob/v0.1.0/BanLat/Quotient.lean}: Constructs quotients of vector
      lattices by order ideals. Defines the quotient lattice
      structure, proves the canonical quotient map is a vector lattice
      homomorphism, develops inequality-lifting lemmas, factors
      positive and lattice homomorphisms through quotients,
      characterises Archimedean quotients by uniform closedness of the
      ideal, and builds normed and Banach lattice quotients by closed
      ideals.

  \item \gh{https://github.com/davidmunozlahoz/banlat/blob/v0.1.0/BanLat/Dual.lean}: Develops the order dual of vector
      lattices and the norm dual of Banach lattices.

  \item \gh{https://github.com/davidmunozlahoz/banlat/blob/v0.1.0/BanLat/Bidual.lean}: Introduces the bidual of a Banach
      lattice and proves that the canonical embedding is a lattice
      homomorphism.

  \item \gh{https://github.com/davidmunozlahoz/banlat/blob/v0.1.0/BanLat/LLexpr.lean}: Defines formal lattice-linear
      expressions in finitely many variables and their evaluation in
      arbitrary vector lattices. It proves that lattice-linear
      identities valid over \R\ hold in every vector lattice.

  \item \gh{https://github.com/davidmunozlahoz/banlat/tree/v0.1.0/BanLat/Substructures}: Collects definitions and basic
      facts about the classical substructures of Banach lattices:
      sublattices, ideals, bands, and projection bands.

      \begin{itemize}
          \item \gh{https://github.com/davidmunozlahoz/banlat/blob/v0.1.0/BanLat/Substructures/Sublattice.lean}: Defines
              vector sublattices. It proves that closure under
              one lattice operation or under absolute value suffices,
              constructs sublattices generated by a set, describes
              them through sup/inf closures and lattice-linear
              combinations, and handles topological closures.

          \item \gh{https://github.com/davidmunozlahoz/banlat/blob/v0.1.0/BanLat/Substructures/Ideal.lean}: Defines
              order ideals in vector lattices. It develops
              generated and principal ideals, the induced norm on
              principal ideals, sums of ideals, closed order ideals,
              and the connection between strong order units and
              principal ideals.

          \item \gh{https://github.com/davidmunozlahoz/banlat/blob/v0.1.0/BanLat/Substructures/Band/Basic.lean}: Defines
              bands and proves basic properties.

          \item \gh{https://github.com/davidmunozlahoz/banlat/blob/v0.1.0/BanLat/Substructures/Band/Lattice.lean}:
              Builds the complete lattice structure on bands ordered
              by inclusion. Intersections give infima, with the whole
              space as top and the zero band as bottom.

          \item \gh{https://github.com/davidmunozlahoz/banlat/blob/v0.1.0/BanLat/Substructures/Band/DisjointComplement.lean}:
              Defines the disjoint complement of a set and proves its
              basic laws. It shows disjoint complements are bands, and
              are closed in normed vector lattices.

          \item \gh{https://github.com/davidmunozlahoz/banlat/blob/v0.1.0/BanLat/Substructures/Band/Generated.lean}:
              Defines the band generated by a set. In Archimedean
              lattices, the band generated by a set is its double
              disjoint complement. Relates weak order units to
              principal bands. Proves that bands are norm closed in
              normed lattices.

          \item \gh{https://github.com/davidmunozlahoz/banlat/blob/v0.1.0/BanLat/Substructures/Band/Projection.lean}:
              Defines projection bands, band projections, and proves
              the basic properties and characterizations.

          \item \gh{https://github.com/davidmunozlahoz/banlat/blob/v0.1.0/BanLat/Substructures/Band/PPP.lean}: Defines
              the projection property and principal projection
              property. It proves that order complete lattices have
              the projection property, $\sigma $-order complete
              lattices have the principal projection property,
              projection property implies principal projection
              property, and these properties imply the Archimedean
              property.

          \item \gh{https://github.com/davidmunozlahoz/banlat/blob/v0.1.0/BanLat/Substructures/Band/Decomposition.lean}:
      Develops infinite decompositions using principal band
        projections under PPP.
\end{itemize}

  \item \gh{https://github.com/davidmunozlahoz/banlat/tree/v0.1.0/BanLat/Operators}: Collects operator theory in vector
      and Banach lattices.
\begin{itemize}
  \item \gh{https://github.com/davidmunozlahoz/banlat/blob/v0.1.0/BanLat/Operators/Hom.lean}: Defines vector lattice
      homomorphism, both as a bundled structure and as a predicate.
      Proves their basic properties. Defines a vector lattice
      equivalence as a bijective vector lattice homomorphism. Defines
      a Banach lattice equivalence as a vector lattice homomorphism
      that is a surjective isometry.

  \item \gh{https://github.com/davidmunozlahoz/banlat/blob/v0.1.0/BanLat/Operators/Positive.lean}: Defines positive
      linear operators. Proves the extension lemma from additive maps
      defined on the positive cone. Proves automatic continuity of
      positive operators from Banach lattices to normed vector
      lattices.

  \item \gh{https://github.com/davidmunozlahoz/banlat/blob/v0.1.0/BanLat/Operators/OrderBounded.lean}: Defines order
      bounded linear maps. It proves positive operators are order
      bounded, order boundedness is stable under algebraic operations,
      order bounded maps from Banach lattices are continuous, and the
      space of order bounded operators is an ordered vector space.

  \item \gh{https://github.com/davidmunozlahoz/banlat/blob/v0.1.0/BanLat/Operators/Regular.lean}: Defines regular
      operators. It proves every regular operator is order bounded.

  \item \gh{https://github.com/davidmunozlahoz/banlat/blob/v0.1.0/BanLat/Operators/RieszKantorovich.lean}: Proves the
      Riesz--Kantorovich formulas for order bounded operators into
      order complete vector lattices.
\end{itemize}

\item \gh{https://github.com/davidmunozlahoz/banlat/tree/v0.1.0/BanLat/Convergences}: The various notions of convergence
    in a vector lattice.
\begin{itemize}
  \item \gh{https://github.com/davidmunozlahoz/banlat/blob/v0.1.0/BanLat/Convergences/Order.lean}: Defines order
      convergence of nets and its basic properties.

  \item \gh{https://github.com/davidmunozlahoz/banlat/blob/v0.1.0/BanLat/Convergences/Uniform.lean}: Defines uniform
      convergence and uniformly complete vector lattices.

  \item \gh{https://github.com/davidmunozlahoz/banlat/blob/v0.1.0/BanLat/OrderContinuous/Basic.lean}: Defines $\sigma
      $-order continuous and order continuous norms. Proves that every order
      continuous Banach lattice is order complete.

  \item \gh{https://github.com/davidmunozlahoz/banlat/blob/v0.1.0/BanLat/OrderContinuous/Ando.lean}: Ando's theorem: a
      Banach lattice has order continuous norm if and only if every
      norm-closed order ideal is a band, and in that situation every
      norm-closed ideal is a projection band.

  \item \gh{https://github.com/davidmunozlahoz/banlat/blob/v0.1.0/BanLat/OrderContinuous/Decomposition.lean}: Proves a
      decomposition theorem for order continuous Banach lattices.

  \item \gh{https://github.com/davidmunozlahoz/banlat/blob/v0.1.0/BanLat/OrderContinuous/MeyerNieberg.lean}: Formalises
      Meyer--Nieberg's characterization of order continuous Banach
      lattices.

  \item \gh{https://github.com/davidmunozlahoz/banlat/blob/v0.1.0/BanLat/OrderContinuous/Nakano.lean}: Formalises
      Nakano's theorem: a Banach lattice is order continuous if and
      only if it is $\sigma $-order complete and $\sigma $-order
      continuous if and only if every increasing order bounded
      sequence converges.
\end{itemize}

\item \gh{https://github.com/davidmunozlahoz/banlat/tree/v0.1.0/BanLat/AMSpace}: The theory of AM-spaces.
\begin{itemize}
  \item \gh{https://github.com/davidmunozlahoz/banlat/blob/v0.1.0/BanLat/AMSpace/Basic.lean}: Defines AM-spaces and
      AM-spaces with unit. Principal ideals in Archimedean vector
      lattices, equipped with the induced norm, are AM-spaces with
      unit.

  \item \gh{https://github.com/davidmunozlahoz/banlat/blob/v0.1.0/BanLat/AMSpace/Max.lean}: Proves the AM-norm formula
      for general nonnegative elements (not necessarily disjoint).

  \item \gh{https://github.com/davidmunozlahoz/banlat/blob/v0.1.0/BanLat/AMSpace/Maximal.lean}: Proves that every proper
      order ideal lies inside a maximal proper order ideal.

  \item \gh{https://github.com/davidmunozlahoz/banlat/blob/v0.1.0/BanLat/AMSpace/Characters.lean}: Defines characters on
      an AM-space with unit as functionals that are lattice
      homomorphisms and send the unit to $1$. Proves a key fact about
      characters: positive elements attain their norm at a character.

  \item \gh{https://github.com/davidmunozlahoz/banlat/blob/v0.1.0/BanLat/AMSpace/Kakutani.lean}: Proves Kakutani's
      representation theorem for AM-spaces.

  \item \gh{https://github.com/davidmunozlahoz/banlat/blob/v0.1.0/BanLat/AMSpace/Dual.lean}: Shows that the dual of an AM-space is an AL-space.
\end{itemize}

\item \gh{https://github.com/davidmunozlahoz/banlat/tree/v0.1.0/BanLat/ALSpace}: The theory of AL-spaces.
\begin{itemize}
  \item \gh{https://github.com/davidmunozlahoz/banlat/blob/v0.1.0/BanLat/ALSpace/Basic.lean}: Defines an AL-space as a Banach lattice whose norm is additive on disjoint positive sums.

  \item \gh{https://github.com/davidmunozlahoz/banlat/blob/v0.1.0/BanLat/ALSpace/Add.lean}: Strengthens the definition
      to additivity of the norm on the positive cone.

  \item \gh{https://github.com/davidmunozlahoz/banlat/blob/v0.1.0/BanLat/ALSpace/Dual.lean}: Constructs a functional on
      the dual that agrees with the norm on positive elements.

  \item \gh{https://github.com/davidmunozlahoz/banlat/blob/v0.1.0/BanLat/ALSpace/DualALAM.lean}: Shows that the dual of an
      AL-space is an AM-space with unit. Shows that if the dual is an
      AL-space, the original space is an AM-space, and conversely.

  \item \gh{https://github.com/davidmunozlahoz/banlat/blob/v0.1.0/BanLat/ALSpace/OrderContinuous.lean}: Shows every AL-space has order continuous norm.

  \item \gh{https://github.com/davidmunozlahoz/banlat/blob/v0.1.0/BanLat/ALSpace/Sublattice.lean}: Proves that closed vector sublattices of AL-spaces are again AL-spaces with their inherited order and norm structure.

  \item \gh{https://github.com/davidmunozlahoz/banlat/blob/v0.1.0/BanLat/ALSpace/Kakutani.lean}: Proves Kakutani's representation theorem for AL-spaces.
\end{itemize}

\item \gh{https://github.com/davidmunozlahoz/banlat/tree/v0.1.0/BanLat/Examples}: Standard examples of Banach lattices.
\begin{itemize}
  \item \gh{https://github.com/davidmunozlahoz/banlat/blob/v0.1.0/BanLat/Examples/CofK/Basic.lean}: Shows that $C(K)$ is
      a Banach lattice, with the usual norm and pointwise operations.

  \item \gh{https://github.com/davidmunozlahoz/banlat/blob/v0.1.0/BanLat/Examples/CofK/Dual.lean}: Proves that the dual of $C(K)$,
      as a Banach lattice, is $M(K)$.

  \item \gh{https://github.com/davidmunozlahoz/banlat/blob/v0.1.0/BanLat/Examples/Lp/Basic.lean}: Shows that $L_p(\mu )$
      is a Banach lattice, for all $1\le p$.

  \item \gh{https://github.com/davidmunozlahoz/banlat/blob/v0.1.0/BanLat/Examples/Lp/Sublattice.lean}: Characterizes
      sublattices of $L_p(\mu )$, for $\mu $ a finite measure and
      $1\le p<\infty $, containing the constant $1$ function.

  \item \gh{https://github.com/davidmunozlahoz/banlat/blob/v0.1.0/BanLat/Examples/SignedMeasure/Basic.lean}: Proves that
      finite signed measures on a measurable space form a Banach
      lattice.

  \item \gh{https://github.com/davidmunozlahoz/banlat/blob/v0.1.0/BanLat/Examples/MofK/Basic.lean}: Defines $M(K)$, the
      space of regular signed Borel measures on a compact Hausdorff
      space $K$, as a closed vector sublattice of the Banach lattice
      of finite signed measures on $K$.

  \item \gh{https://github.com/davidmunozlahoz/banlat/blob/v0.1.0/BanLat/Examples/MofK/ALspace.lean}: Shows $M(K)$ is an AL-space.

  \item \gh{https://github.com/davidmunozlahoz/banlat/blob/v0.1.0/BanLat/Examples/MofK/Atom.lean}: Characterises atoms
      in $M(K)$ as positive scalar multiples of Dirac measures.

  \item \gh{https://github.com/davidmunozlahoz/banlat/blob/v0.1.0/BanLat/Examples/MofK/Band.lean}: Describes principal
      bands in $M(K)$.

  \item \gh{https://github.com/davidmunozlahoz/banlat/blob/v0.1.0/BanLat/Examples/MofK/Decomposition.lean}: Develops the
      decomposition of an element of $M(K)$ into its atomic and
      discrete parts.

  \item \gh{https://github.com/davidmunozlahoz/banlat/blob/v0.1.0/BanLat/Examples/MofK/L1space.lean}: Proves that $M(K)$
      is lattice isometric to $L_1(\mu )$, for some measure $\mu $.
\end{itemize}

\item \gh{https://github.com/davidmunozlahoz/banlat/tree/v0.1.0/BanLat/Free}
\begin{itemize}
  \item \gh{https://github.com/davidmunozlahoz/banlat/blob/v0.1.0/BanLat/Free/FVL.lean}: Constructs the free vector
      lattice on a type of generators $\alpha $ as a concrete vector sublattice
      of real-valued functions on the space $\alpha \to \R$.

  \item \gh{https://github.com/davidmunozlahoz/banlat/blob/v0.1.0/BanLat/Free/FVLv.lean}: Constructs the free vector lattice generated by a Banach space as a sublattice of continuous functions on the dual unit ball.

  \item \gh{https://github.com/davidmunozlahoz/banlat/blob/v0.1.0/BanLat/Free/FBL.lean}: Constructs the free Banach lattice over a Banach space.
\end{itemize}

\item \gh{https://github.com/davidmunozlahoz/banlat/tree/v0.1.0/BanLat/LocallySolid}: The theory of locally solid topologies.
\begin{itemize}
  \item \gh{https://github.com/davidmunozlahoz/banlat/blob/v0.1.0/BanLat/LocallySolid/Basic.lean}: Defines locally solid
      topological vector lattices, and proves basic properties.

  \item \gh{https://github.com/davidmunozlahoz/banlat/blob/v0.1.0/BanLat/LocallySolid/Completion.lean}: Builds the
      completion of a Hausdorff locally solid vector lattice, endowing
      it with the structure of a Hausdorff locally solid vector
      lattice.

  \item \gh{https://github.com/davidmunozlahoz/banlat/blob/v0.1.0/BanLat/LocallySolid/Lebesgue.lean}: Defines Lebesgue,
      $\sigma $-Lebesgue, and pre-Lebesgue properties for locally solid vector lattices.

  \item \gh{https://github.com/davidmunozlahoz/banlat/blob/v0.1.0/BanLat/LocallySolid/LocallyConvexSolid.lean}: Defines
      locally convex-solid vector lattices as locally solid vector
      lattices with locally convex topology. It proves the
      neighbourhoods of zero admit a basis of sets that are at the
      same time solid and convex.

  \item \gh{https://github.com/davidmunozlahoz/banlat/blob/v0.1.0/BanLat/LocallySolid/WithSeminorms.lean}: Introduces
      lattice seminorms. Characterises locally convex-solid linear topologies as those induced by families of lattice seminorms.
\end{itemize}

\item \gh{https://github.com/davidmunozlahoz/banlat/tree/v0.1.0/BanLat/Preliminaries}: Preliminary facts required for
    the library. For now, they are only facts about measure theory.
\begin{itemize}
  \item \gh{https://github.com/davidmunozlahoz/banlat/blob/v0.1.0/BanLat/Preliminaries/SignedMeasure.lean}: Collects
      facts about signed measures. It proves that a sequence of
      measures that is Cauchy for the total variation norm converges.

  \item \gh{https://github.com/davidmunozlahoz/banlat/blob/v0.1.0/BanLat/Preliminaries/Regularity.lean}: Defines
      regularity of signed measures on compact Hausdorff spaces by
      regularity of their total variation.

  \item \gh{https://github.com/davidmunozlahoz/banlat/blob/v0.1.0/BanLat/Preliminaries/HasNoAtoms.lean}: Proves that, if
      $\mu $ is a non-atomic measure and $A$ is a $\mu $-measurable
      set, then for every $0\le \lambda \le \mu (A)$ there exists
      $B\subseteq A$ measurable such that $\mu (B)=\lambda $.
\end{itemize}

\item \gh{https://github.com/davidmunozlahoz/banlat/tree/v0.1.0/BanLat/Tactic}: New tactics implemented in the library.
\begin{itemize}
    \item \gh{https://github.com/davidmunozlahoz/banlat/blob/v0.1.0/BanLat/Tactic/LLexpr.lean}: Implements the
        \gh[llarith]{https://github.com/davidmunozlahoz/banlat/blob/ef1bc06da46e90ed7934b35c11be07d7395bdaec/BanLat/Tactic/LLexpr.lean\#L474} tactic for lattice-linear identities and non-strict inequalities.
\end{itemize}
\end{itemize}

\section{Extending the library}\label{sec:extension}

The following are some natural directions in which the library could
be extended.

\begin{description}
    \item[Locally solid topologies] Several results about convergence
        in Banach lattices (such as the Meyer-Nieberg and Andô
        theorems in order continuous Banach lattices) can be
        extended to the more general context of locally solid topologies.
    \item[Order continuous operators] Develop the theory of ($\sigma
        $-)order
        continuous operators. In particular, prove Ogasawara's
        theorem: the families of order continuous and $\sigma $-order
        continuous operators are band projections on the space of
        order bounded operators, whenever the codomain is order
        complete.
    \item[Orthomorphisms, \falg s, and universal completions] Develop
        the theory of orthomorphisms. This will also require
        introducing \falg s, and constructing universal completions.
    \item[Nakano theory] Define carriers, null ideals, and prove
        Nakano's theorem relating the carrier and null ideals of
        disjoint order continuous functionals.
    \item[Free Banach lattices] Develop further the theory of free
        Banach and vector lattices, proving properties about
        these spaces.
    \item[Atomic and order continuous Banach lattices] Develop the
        theory of Banach lattices that are at the same time atomic and
        order continuous.
    \item[Enrich the examples] Prove further properties about $C(K)$
        and $L_p(\mu )$: characterize their ideals, bands, and
        projection bands; characterize when $C(K)$ is ($\sigma $)-order
        complete, etc. In general, lay down the bridges between Banach
        lattice theory and topology/measure theory.
    \item[Positively homogeneous function calculus] Construct
        the positively homogeneous function calculus on uniformly complete
        vector lattices. Define $p$-sums.
    \item[Convex/concave Banach lattices and operators] Develop the
        theory of convexity of Banach lattices and operators.
\end{description}

\section*{Acknowledgements}

We would like to thank M.\ A.\ Taylor and P.\ Tradacete for their
feedback and continued support. We are also indebted to J.\ de Dios
for his feedback, and for suggesting the idea of creating the
\gh[llarith]{https://github.com/davidmunozlahoz/banlat/blob/ef1bc06da46e90ed7934b35c11be07d7395bdaec/BanLat/Tactic/LLexpr.lean\#L474}
tactic. Finally, we thank E.\ Bilokopytov and M.\ de Jeu for feedback on
\cref{sec:insights}.

Research supported by an FPI–UAM 2023 contract (funded by Universidad
Autónoma de Madrid), by grants PID2024-162214NB-I00 and
CEX2023-001347-S (funded by MCIN/AEI/10.13039/501100011033), and by
the CSIC cooperation grant COOPB25033 under the i-COOP program.

\section*{AI disclosure}

GPT-5.5 and Claude Opus 4.6 were used to write the code in \banlat. In the process of writing the paper,
GPT-5.5 was only used to search for typos and other obvious errors in
final versions of the text, and to generate the list of files with
their docstrings (which were later rewritten) in \cref{sec:overview}.

\emergencystretch=1em
\printbibliography

\end{document}